\documentclass{amsart}
\usepackage{indentfirst,enumerate,cite,amssymb,amsfonts,amsmath,amsthm,mathrsfs,dsfont}
\usepackage{color}
\usepackage[colorlinks=true,linkcolor=blue,citecolor=blue]{hyperref}

\usepackage[onehalfspacing]{setspace}

\numberwithin{equation}{section}
\theoremstyle{plain}
\newtheorem{thm}{Theorem}[section]
\newtheorem{prop}[thm]{Proposition}

\newtheorem{lemma}[thm]{Lemma}
\newtheorem{defi}[thm]{Definition}
\newtheorem{remark}[thm]{Remark}
\newtheorem{ex}[thm]{Example}

\newcommand{\HS}{{\mathtt{HS}}}
\newcommand{\Tr}{\mbox{\emph{Tr}}}
\newcommand{\N}{{\mathbb{N}}}
\newcommand{\Q}{\mathbb{Q}}
\newcommand{\Z}{\mathbb{Z}}
\newcommand{\R}{\mathbb{R}}
\newcommand{\C}{\mathbb{C}}

\newcommand{\St}{{\mathbb S}^3}
\newcommand{\Sth}{\widehat{{\mathbb S}^3}}
\newcommand{\TS}{\mathbb{T}^1\times\St}

\newlength{\dhatheight}
\newcommand{\doublehat}[1]{%
	\settoheight{\dhatheight}{\ensuremath{\hat{#1}}}
	\addtolength{\dhatheight}{-0.15ex}
	\widehat{\vphantom{\rule{5pt}{\dhatheight}}%
		\smash{\widehat{#1}}}}

\begin{document}

%
%
%
%
%
%
%
%
%

\title[Global analytic hypoellipticity on $\mathbb{T}^1\times\mathbb{S}^3$]{Global analytic hypoellipticity for a class of \\ evolution operators on $\mathbb{T}^1\times\mathbb{S}^3$}



\author[Alexandre Kirilov]{Alexandre Kirilov}
\address{
	Universidade Federal do Paran\'{a}, 
	Departamento de Matem\'{a}tica,
	C.P.19096, CEP 81531-990, Curitiba, Brazil
}
\email{akirilov@ufpr.br}


\author[Ricardo Paleari]{Ricardo Paleari da Silva}
\address{
	Universidade Federal do Paran\'{a},
	Programa de P\'os-Gradua\c c\~ao de Matem\'{a}tica, 
	C.P.19096, CEP 81531-990, Curitiba, Brazil
}
\email{ricardopaleari@gmail.com}


\author[Wagner de Moraes]{Wagner A. A. de Moraes}
\address{
	Universidade Federal do Paran\'{a},
	Programa de P\'os-Gradua\c c\~ao de Matem\'{a}tica,
	C.P.19096, CEP 81531-990, Curitiba, Brazil
}
\email{wagneramat@gmail.com}



\begin{abstract}

In this paper, we present necessary and sufficient conditions to have global analytic hypoellipticity for a class of first-order operators defined on $\mathbb{T}^1 \times \mathbb{S}^3$.  In the case of real-valued coefficients, we prove that an operator in this class is conjugated to a constant-coefficient operator satisfying a Diophantine condition, and that such conjugation preserves the global analytic hypoellipticity. In the case where the imaginary part of the coefficients is non-zero, we show that the operator is globally analytic hypoelliptic if the Nirenberg-Treves condition ($\mathcal{P}$) holds, in addition to an analytic Diophantine condition.
	
\end{abstract}
\maketitle

\section{Introduction}

In this work, we are concerned with the global analytic hypoellipticity of first-order operators of  form
\begin{equation}\label{EQ:L}
P = \partial_t + (a+ib)(t) \partial_0 + q,
\end{equation}
where $q \in \mathbb{C}$, $a$ and $b$ are real-valued and real analytic functions on $\mathbb{T}^1 = \R/(2\pi\Z)$, and $\partial_0$ is the left-invariant vector field on $\St$ known as the neutral operator.   

The operator $P$ is said to be globally analytic hypoelliptic (GAH) in $\TS$ if the conditions $u \in \mathcal{D}' (\TS)$ and $Pu \in C^\omega(\TS)$ imply that $u \in C^\omega(\TS).$ 

\medskip
The global analytic (and smooth) hypoellipticity of vector fields and systems of vector fields has been extensively studied on tori. Within these studies, we cite as most important and inspiring for this project:   \cite{AGKM18,AKM,BNZ00,BZ05,BZ08,Ber94,Ber99,bergamasco2017existence,CH98,GW72,HP00,Hou79,Hou82,Pet11}.

In the specific case of global analytic hypoellipticity in the torus $\mathbb{T}^2$, Berga\-masco proved  in \cite{Ber99} that $\partial_t+(a(t)+ib(t)) \partial_x$  is (GAH) if and only if either $b(t)$ does not change sign or $b\equiv0$ and the real number $a_0=\frac{1}{2\pi} \int_0^{2\pi}a(t)dt$ is neither rational nor an exponential Liouville number. 

Let us recall that an irrational number $\lambda$ is said to be an exponential Liouville number if there exists $\epsilon > 0$ such that the inequality $|\lambda -p/q|\leq e^{-\epsilon q}$ has infinitely many rational numbers $p/q$. 

Next, in \cite{BZ08}, Bergamasco and Zani proved that, if there exists a non-singular, globally analytic hypoelliptic vector field $L$ on a compact surface $M$, then $M$ is real analytically diffeomorphic to $\mathbb{T}^2$ and, either the Nirenberg-Treves condition ($\mathcal{P}$) holds in $M$, or there are coordinates on which we can write $L = g(x, t)(\partial_t + \lambda \partial_x)$, where $g\neq 0$ everywhere and $\lambda$ is a real number which is neither rational nor exponential-Liouville.

Our results in $\TS$ and the characterization given by Bergamasco and Zani in dimension 2 above, suggest the existence of an analytic version of the famous Greenfield's and Wallach's conjecture, see \cite{F08}. In this way, our operators are essentially low-order perturbations of left-invariant vector fields.

In the case of constant-coefficients operators our main result is as follows.

\begin{thm}\label{const-coeff-intr}
	 Let $c,q\in\C$. The operator $L=\partial_t + c\partial_0 + q$ is globally analytic hypoelliptic if and only if for all $B>0$, there is  $K_B>0$ such that for all $k, \ell \in \Z$,
	\begin{equation}\label{ADC3-intr}
	\left|k+\tfrac{1}{2}\ell c-iq\right| \geq K_B e^{-B(|k|+|\ell|)}. \tag{ADC$_3$}
	\end{equation}
\end{thm}

For example, writing $c=a+ib$ if $b\neq 0$ and ${\rm Re}(q)/ b\notin \tfrac{1}{2}\mathbb{Z}$, then $L$ is  globally analytic hypoelliptic. And when $b = 0$ and $iq \in \mathbb{Z}$, then $L$ is  globally analytic hypoelliptic if and only if $a$ is neither rational nor an exponential Liouville number. 

This theorem follows directly from  Proposition \ref{Ninfinite}, Remark \ref{finiteimpliesempty}, Proposition \ref{gahconst} and Lemma \ref{diophantine2}, while  the details of the above example are given in Example \ref{ex_const-coef}.

For the general case \eqref{EQ:L}, we introduce the following notation:
\begin{equation*}\label{P0q-introd}
P_0 \doteq  \partial_t + (a_0+ib_0)\partial_0 + q,
\end{equation*}
where
\begin{equation*}\label{c0-introd}
a_0= \frac{1}{2\pi} \int_0^{2 \pi} a(s)ds \mbox{ \ and \ } b_0= \frac{1}{2\pi} \int_0^{2 \pi} b(s)ds.
\end{equation*}

\begin{thm}\label{mainthm-intr} 
	$P = \partial_t +(a+ib)(t)\partial_0 + q$ is (GAH) if and only if one of the following conditions holds:
	\begin{enumerate}
		\item if $b \not\equiv 0$ then $b$ does not change sign; and either 
		 $$\dfrac{{\rm Re}(q)}{b_0} \notin \tfrac{1}{2}\mathbb{Z} \mbox{ \ or \ } {\rm Im}(q)+{\rm Re}(q)\dfrac{a_0}{b_0} \notin \mathbb{Z}.$$
		\item if $b \equiv 0$, then $P_0$ is (GAH). 
	\end{enumerate}
\end{thm}

This theorem follows from Theorems \ref{Lisgah}, \ref{equivalence2}, \ref{notgah}, \ref{relation}, \ref{mainthm}
and Remark \ref{last-rem}. Part of our proofs follows from the ideas used for the smooth case, in \cite{AGKM18,AK19,bergamasco2017existence}, which rely heavily on the use of cut-off functions. One of the difficulties of adapting such arguments is that in the analytic case there are no such functions. To overcome this problem we drew on ideas used in \cite{BNZ00} (to construct singular solutions) and in \cite{Sjostrand} (to analyze the asymptotic behavior of a sequence of integrals at infinity).

\bigskip
\section{Fourier analysis on $\mathbb{T}^1\times\mathbb{S}^3$}

In this section, we introduce the notations and recall the main results necessary for the development of this study, which can be found in reference \cite{livropseudo}.  

The three-dimensional sphere $\St$ is a Lie group with respect to the quaternionic product of $\mathbb{R}^{4}$, and it is isomorphic as a Lie group to the set of unitary $2\times 2$ matrices of determinant one $\mbox{SU}(2)$, with the usual matrix product. Let $\Sth$ be the unitary dual of $\St$, that is, the set of equivalence classes $[\textsf{t}^\ell]$ of continuous irreducible unitary representations $\textsf{t}^\ell:\St\to \C^{(2\ell+1)\times (2\ell+1)}$, $\ell\in\frac12\N_{0}$, of matrix-valued functions satisfying $\textsf{t}^\ell(xy) = \textsf{t}^\ell(x)\textsf{t}^\ell(y)$ and $\textsf{t}^\ell(x)^{*}=\textsf{t}^\ell(x)^{-1}$ for all
$x,y\in\St$.

We will use the standard convention of enumerating the matrix elements $\textsf{t}^\ell_{mn}$ of $\textsf{t}^\ell$ using indexes $m,n$ ranging from $-\ell$ to $\ell$ with step one, i.e. we have $-\ell\leq m,n\leq\ell$
with $\ell-m, \ell-n\in\N_{0}.$

The Fourier coefficient of a function $f\in C^{\infty}(\St)$, at $\ell\in\frac12\N_{0}$, is given by
$$
\widehat{f}(\ell)\doteq \int_{\St} f(x) \textsf{t}^{\ell}(x)^{*}dx \in {\C^{(2\ell+1)\times (2\ell+1)}},
$$
where the integral is taken with respect to the Haar measure on $\St$. This definition is naturally extended to distributions, and the Fourier series becomes
$$
f(x)=\sum_{\ell\in\frac12\N_{0}} (2\ell+1) \mbox{Tr}\left( \textsf{t}^\ell(x)\widehat{f}(\ell)\right),
$$
with the Plancherel's identity assuming the form
\begin{equation}\label{EQ:Plancherel}
\|f\|_{L^{2}(\St)}=\left(\sum_{\ell\in\frac12\N_{0}} (2\ell+1) 
	\|\widehat{f}(\ell)\|_{\HS}^{2}\right)^{1/2},
\end{equation}
where 
$ \|\widehat{f}(\ell)\|_{\HS}^{2}=\emph{\Tr}(\widehat{f}(\ell)\widehat{f}(\ell)^*)$ is the
Hilbert--Schmidt norm of the matrix $\widehat{f}(\ell)$.

Smooth functions and distributions on $\St$ can be characterized in terms of
their Fourier coefficients in the following way: 
$$
f\in C^{\infty}(\St)\Longleftrightarrow
\forall N \;\exists C_{N}>0 \textrm{ such that }
\|\widehat{f}(\ell)\|_{\HS}\leq C_{N} (1+\ell)^{-N}, \forall \ell\in\frac12\N_{0}
$$ 
and
$$
u\in  \mathcal{D}'(\St)
\Longleftrightarrow
\exists M\in\mathbb{N} \;\exists C>0 \textrm{ such that }
\|\widehat{u}(\ell)\|_{\HS}\leq C(1+\ell)^{M}, \forall \ell\in\frac12\N_{0}.
$$

Given an operator $T:C^{\infty}(\St)\to C^{\infty}(\St)$, we define its matrix symbol by
$$\sigma_{T}(x,\ell)\doteq \textsf{t}^\ell(x)^{*} (T \textsf{t}^\ell)(x)\in {\C^{(2\ell+1)\times (2\ell+1)}},$$ where
$T \textsf{t}^\ell$ means that we apply $T$ to the matrix components of $\textsf{t}^\ell(x)$.
In this case we have 
\begin{equation}\label{EQ:T-op}
Tf(x)=\sum_{\ell\in\frac12\N_{0}} (2\ell+1) \mbox{Tr}\left({\textsf{t}^\ell(x)\sigma_{T}(x,\ell)\widehat{f}(\ell)}\right).
\end{equation}

The correspondence between operators and symbols is one-to-one, and we will write $T_{\sigma}$ for the operator given by \eqref{EQ:T-op} corresponding to the symbol $\sigma(x,\ell)$.
The properties and corresponding symbolic calculus of the quantization \eqref{EQ:T-op} were extensively studied in \cite{livropseudo,RT13}.

For each $f \in L^1(\mathbb{T}^1\times \St)$ and $\ell \in \tfrac{1}{2}\N_0$, we define the $mn$--component of the partial Fourier coefficient of $f$ with respect to the $x$ variable as
$$
\widehat{f}(t,\ell)_{mn} = \int_{\St} f(t,x) \overline{\textsf{t}^{\ell}(x)_{nm}}\, dx \in L^1(\mathbb{T}^1)
$$
and, for each $k\in\Z$, we denote by $\doublehat{\,f\,}(k,\ell)_{mn}$ the $k$--th Fourier coefficient of the function  $\widehat{f}(\cdot,\ell)_{mn} $.

We have the following characterizations of the spaces $C^\infty(\mathbb{T}^1\times\St)$, $\mathcal{D}'(\mathbb{T}^1\times\St)$ and $C^\omega(\mathbb{T}^1\times\St)$ (see \cite{KMR19d} and \cite{KMR19}).

\begin{prop}\label{caracfunc}
Let $\{\widehat{f}(\:\cdot \:, \ell)_{mn} \}$ be a sequence of functions on $\mathbb{T}^1$ and define
	$$
	f(t,x) \doteq  \sum_{\ell\in\tfrac{1}{2}\N_0} (2\ell+1) \sum_{m,n}\widehat{f}(t, \ell)_{mn}\textsf{t}^{\ell}(x)_{nm}, \ (t,x)\in \mathbb{T}^1\times\St.
	$$ 
	Then $f \in C^\infty(\mathbb{T}^1\times\St)$ if and only if $\widehat{f}(\: \cdot\:,\ell)_{mn} \in C^\infty(\mathbb{T}^1)$, for all $\ell\in\tfrac{1}{2}\N_0$, $-\ell \leq m,n \leq \ell$ and for every multi-index $\beta$ and $N>0$, there exists $C_{\beta N} > 0$ such that
	\begin{equation*}\label{carac}
	\bigl|\partial^\beta \widehat{f}(t,\ell)_{mn}\bigl| \leq C_{\beta N} (1+\ell)^{-N}, \quad \forall t\in \mathbb{T}^1, \ \ell\in\tfrac{1}{2}\N_0, \ -\ell\leq m,n \leq \ell.
	\end{equation*}
\end{prop}
\begin{prop}\label{caracdist}
Let $\bigl\{\widehat{u}(\: \cdot \:,\eta)_{rs} \bigr\}$ be a sequence of distributions on $\mathbb{T}^1$ and define 
\begin{equation*}
	u=\sum_{\ell\in\tfrac{1}{2}\N_0} (2\ell+1) \sum_{m,n}\widehat{u}(\:\cdot\:,\eta)_{mn}\textsf{t}^{\ell}_{nm}.
\end{equation*} 
	Then $u \in \mathcal{D}'(\mathbb{T}^1\times \St)$ if and only if there are $K \in \N$ and $C>0$ such that
	\begin{equation}\label{pk}
	\bigl| \left\langle \widehat{u}(\cdot,\ell)_{mn},\varphi\right\rangle \bigr| \leq C \, p_K(\varphi) (\ell+1)^{K}, 
	\end{equation} 
	for all $\varphi \in C^{\infty}(\mathbb{T}^1)$ and $\ell\in\tfrac{1}{2}\N_0$, where $p_K(\varphi) \doteq  \sum\limits_{\beta \leq K} \left\|\partial^\beta\varphi\right\|_{L^\infty(\mathbb{T}^1)}.$
\end{prop}

\begin{prop}\label{functionanaly} Let $f \in C^\infty(\mathbb{T}^1\times \St)$. We have that $f\in C^\omega(\mathbb{T}^1\times \St)$ if and only if $\widehat{f}(\cdot,\ell)_{mn} \in C^\omega(\mathbb{T}^1)$ for every $\ell \in \tfrac{1}{2}\N_0$, $-\ell \leq m,n \leq \ell$ and there are $h,C,B>0$ such that
$$|\partial^\beta\widehat{f}(t,\ell)_{mn}| \leq Ch^\beta \beta! \, e^{-B \ell},$$
for all multi-index $\beta$, $t \in \mathbb{T}^1$, $\ell \in \tfrac{1}{2}\N_0$, $-\ell \leq m,n \leq \ell$. 
\end{prop}
\begin{remark}\label{noderivative}
When the functions  $t\in\mathbb{T}^1\mapsto\widehat{f}(\cdot,\ell)_{mn}$ are real analytic, it follows from Cauchy's integral formula that it is enough to obtain estimates for $\beta=0$ in the last proposition above.	
\end{remark}

\begin{defi}
	An operator $P:\mathcal{D}' (\mathbb{T}^1\times \St) \to \mathcal{D}' (\mathbb{T}^1\times \St)$ is said to be globally analytic hypoelliptic (GAH) if the conditions $u \in \mathcal{D}' (\mathbb{T}^1\times \St)$ and $Pu \in C^\omega(\mathbb{T}^1\times \St)$ imply that $u \in C^\omega(\mathbb{T}^1\times \St)$. 
	
	When the conditions $u \in \mathcal{D}'(\mathbb{T}^1\times \St)$ and $Pu \in C^\infty(\mathbb{T}^1\times \St)$ imply that $u \in C^\infty(\mathbb{T}^1\times \St),$ we say that  $P$ is globally hypoelliptic (GH).
	
\end{defi}

\section{Global analytic hypoellipticity and  analytic Diophantine conditions}\label{GAH+ADC}

We will begin the study of global analytic hypoellipticity on $\mathbb{T}^1 \times \St$ by addressing a class of constant-coefficient operators that will play an important role in the development of this article.

Let $c,q \in \C$ be constants, $X$ be a normalized smooth vector field on $\St$ and consider the operator $L=\partial_t + c X+q$ defined on $\mathbb{T}^1 \times \St$.

 By using rotations on $\St$, without loss of generality, we may assume that the vector field $X$ is the operator $\partial_0$ that has the symbol 
$$
\sigma_{\partial_0}(\ell)_{mn}=im\delta_{mn}, \quad \ell \in \tfrac{1}{2}\N_0, \ -\ell\leq m,n\leq \ell, \ \ell-m, \ell-n \in \N_0,
$$ 
where $\delta_{mn}$ is the Kronecker's delta (see \cite{livropseudo}, \cite{RT13}, and \cite{RTW14}).  

Thus, given $c,q \in \C$, we will consider the following operator defined on $\TS$
\begin{equation}\label{Lq}
  L = \partial_t + c \partial_0+q.
\end{equation}

Now, taking the Fourier coefficients separately in each variable in $Lu=f$ (see \cite{KMR19} for more details), we have
$$	i(k+c m-iq) \doublehat{\, u \,}\!(k,\ell)_{mn} = \doublehat{\,f\,}\!(k,\ell)_{mn},	$$
where $k \in \Z$, $\ell \in \frac{1}{2}\N_0$, $-\ell \leq m,n \leq \ell$ and $\ell-m, \ell-n \in \N_0$. 

\medskip
Then we are lead to consider the following set:
\begin{equation*}\label{setN}
\mathcal{N} = \left\{(k,\ell) \in \Z\times \tfrac{1}{2}\N_0; k+cm-iq=0, \textrm{ for some } -\ell \leq m \leq \ell, \ \ell-m\in\N_0 \right\}.
\end{equation*}

\begin{prop}\label{Ninfinite} 
	If the set $\mathcal{N}$ has infinitely many elements, then the operator $L$ defined in \eqref{Lq} is not (GAH).
\end{prop}

\begin{proof} Consider the sequence
$$\doublehat{\,u\,}(k,\ell)_{mn} = \left \{\begin{array}{cl} 1, & \textrm{if } k+{c}m -iq=0, \\ 0, & \textrm{ in any other case.} \end{array} \right. $$
	
Since $|\doublehat{\,u\,}(k,\ell)_{mn}|\leq 1$, by Proposition \ref{caracdist}, this sequence defines a distribution given by $$u(t,x)=\sum_{\ell} (2\ell+1) \sum_{m,n}\widehat{u}(t,\ell)_{mn}\textsf{t}^{\ell}_{nm} \in \mathcal{D}'(\mathbb{T}^1 \times \St),$$ 
which satisfies the equation $L u=0$.

The elements of this sequence do not, however, decay as required in Propositions \ref{caracfunc} and \ref{functionanaly}, because the set $\mathcal{N}$ has infinitely many elements, therefore $u \notin C^\infty(\mathbb{T}^1\times \St)$ and $L$ not (GAH).

\end{proof}

\begin{remark}\label{finiteimpliesempty} If $\mathcal{N} \neq \varnothing$, then $\mathcal{N}$ is infinite. 
	
	Indeed, if $k+cm-iq=0$, for some $(k,\ell_0) \in \mathbb{Z}\times \frac{1}{2}\mathbb{N}_0$ and  $-\ell_0\leq m \leq \ell_0$ then, by the disposition of the indexes on the matrix representations $\textsf{t}^{\ell}$, the same index $m$ will appear in $\ell+n$, for all $n \in \N$, and this means that we have $(k,\ell+n) \in \mathcal{N}$, for all $n\in \N$.   
\end{remark}

Next, we present a Diophantine condition like and the characterization of the global analytic hypoellipticity of the operator $L$, which can be obtained from Theorem 6.1 in \cite{KMR19c}.

\begin{defi} We say that the operator $L$ satisfies the analytic Diophantine condition (ADC)  if 
	for all $B>0$ there exists a constant $K_B>0$ such that for all $k\in \Z$, $\ell \in \tfrac{1}{2}\N_0$, $-\ell \leq m \leq \ell$, $\ell-m \in \N_0$, 
\begin{equation}\label{ADC} 
	|k+cm-iq| \geq K_Be^{-B(k+\ell+1)}, \tag{ADC}
\end{equation}
whenever $k+cm-iq \neq 0$.
\end{defi}

\begin{prop}\label{gahconst}
The operator $L=\partial_t+c\partial_0+q$ is (GAH) if and only if the set $\mathcal{N}$ is finite and $L$ satisfies the  \eqref{ADC}  condition.
\end{prop}

The next lemma gives us two equivalent forms of the  \eqref{ADC} condition, which will be useful to construct examples of global analytic hypoelliptic operators.

\begin{lemma}\label{diophantine2}
	The following analytic Diophantine conditions are equivalent to the \eqref{ADC}  condition:
	\begin{enumerate}[(i)]
		\item for all $B>0$, there is $K_B>0$ such that for all $k \in \Z$ and $\ell \in \tfrac{1}{2}\Z$,
		\begin{equation}\label{ADC2}
		|k+c\ell-iq| \geq K_B e^{-B(|k|+|\ell|)}, \tag{ADC$_2$}
		\end{equation}
		whenever $k+c\ell -iq \neq 0$.
		\item	 for all $B>0$, there is  $K_B>0$ such that for all $k, \ell \in \Z$,
		\begin{equation}\label{ADC3}
		\left|k+\tfrac{c}{2}\ell-iq\right| \geq K_B e^{-B(|k|+|\ell|)},  \tag{ADC$_3$}
		\end{equation}
		whenever $k+\tfrac{c}{2}\ell -iq \neq 0$.
	\end{enumerate}
\end{lemma}
\begin{proof}
	First, let us prove the equivalence between \eqref{ADC} and \eqref{ADC2}. Assume that \eqref{ADC} holds and let $\tau \in \Z$ and $\rho \in \tfrac{1}{2}\Z$ such that $\tau+c\rho-iq \neq 0$. Taking $k=\tau$, $\ell = |\rho|$ and $m=\rho$ in \eqref{ADC} we have
	$$
	|\tau+c\rho-iq| \geq K_Be^{-B(|\tau|+|\rho|)}.
	$$
	On the other hand, if \eqref{ADC2} holds, let $\tau \in \Z$, $\rho \in \tfrac{1}{2}\N_0$ and $-\rho \leq m \leq \rho$ such that $\rho-m\in\N_0$ and $\tau+cm-iq\neq0$. Taking $k=\tau$ and $\ell =m$ in \eqref{ADC2}, since $|m| \leq \rho$, we have
	$$
	|\tau+cm-iq| \geq K_Be^{-B(|\tau|+|m|)} \geq K_Be^{-B(|\tau|+\rho)}.
	$$
	
	Now, to prove that \eqref{ADC2} is equivalent to \eqref{ADC3},  assume that \eqref{ADC2} is true and let $\tau, \rho \in \Z$ such that $k+\tfrac{c}{2}\rho -iq\neq 0$. Taking $k=\tau$ and $\ell=\tfrac{\rho}{2}$ in \eqref{ADC2}, we have 
	$$
	|\tau + \tfrac{c}{2}\rho - iq| \geq K_Be^{-B(|\tau|+|\rho|/2)} \geq K_Be^{-B(|\tau|+|\rho|)}.
	$$
	
	Finally, assuming that \eqref{ADC3}  holds, let $\tau \in \Z$, $\rho \in \tfrac{1}{2}\Z$ such that $\tau+c\ell-iq\neq 0$. We can write $\rho=\tfrac{r}{2}$, for some $r \in \Z$. Taking $k=\tau$ and $\ell=r$ in \eqref{ADC3}, we have
	$$
	|\tau+c\rho-iq|=|\tau+\tfrac{c}{2}r -iq| \geq K_B e^{-B(|k|+|r|)} \geq K_Be^{-2B(|k|+|\rho|)}.
	$$
	
\end{proof}

\begin{ex} \label{ex_const-coef}
	Writing $c=a+ib \in \mathbb{C}$ we have 
  $$
  k+cm-iq = (k+am+{\rm Im}(q)) +i(mb-{\rm Re}(q)).
  $$

Thus, if $b\neq 0$ and\, ${\rm Re}(q)/ b\notin \tfrac{1}{2}\mathbb{Z}$, then 
$$|k+cm-iq| \geq  |b(m-{\rm Re}(q)/b)| \geq K = \mbox{constant}.$$ 
Hence, the set $\mathcal{N}$ is empty and the \eqref{ADC} condition  is satisfied, which implies, by Proposition \ref{gahconst}, that $L$ is (GAH). 

If $b\neq 0$ and $\emph{\textrm{Re}}(q)/b \in \tfrac{1}{2}\mathbb{Z}$ we have two cases to consider. When ${\rm Im}(q)+\emph{\textrm{Re}}(q)a/b \in \Z$ the set $\mathcal{N}$ has infinitely many elements and so, by Proposition \ref{Ninfinite}, $L$ is not (GAH). When $\emph{\textrm{Im}}(q)+\emph{\textrm{Re}}(q)a/b \not\in \Z$, the set $\mathcal{N}$ is empty and we have $\delta=\inf\limits_{k\in\Z}\{k+\emph{\textrm{Re}}(q)a/b+\emph{\textrm{Im}}(q)\}>0$. Hence,
$$
|k+cm-iq| \geq \max\{|b|, \delta\} =\mbox{constant} >0,
$$
which implies that $L$ is (GAH).

Similarly, if $b=0$ and ${\rm Re}(q) \neq0$ we obtain
$$|k+cm-iq| \geq  |{\rm Re}(q)| \geq K = \mbox{constant},$$ 
and so $L$ is (GAH).

However, when ${\rm Re}(q)= b = 0$, by the \eqref{ADC3} condition, the global analytic hypoellipticity of $L$ depends on the approximations $ \left|k+\tfrac{a}{2}\ell+{\rm Im}(q)\right|$, with $k,\ell\in \Z$. 

In the special case where $b = 0$ and $iq \in \mathbb{Z}+\tfrac{a}{2}\Z$, then 
$$\left|k+\tfrac{c}{2}\ell-iq\right| = \left|\widetilde{k}+\tfrac{a}{2}\widetilde{\ell}\right|,$$ 
with $\widetilde{k},\widetilde{\ell}\in \Z.$ 

Therefore, the operator $L = \partial_t+a \partial_0 + {q}$  is  (GAH)  if and only if $a$ is neither rational nor an exponential Liouville number.

Finally, when ${\rm Re}(q)= b = 0$ and $iq\notin \Z+\tfrac{a}{2}\Z$ the set $\mathcal{N}$ is empty  and we can also guarantee the existence of $K>0$ such that 
  $$
\left|k+\tfrac{a}{2}\ell+{\rm Im}(q)\right|\geq K = \mbox{constant},
$$
when $a \in \Q$ because, in this case, the set $\Z+\tfrac{a}{2}\Z$ is discrete. 

Then the \eqref{ADC3} condition is satisfied, which implies that $L$ is (GAH). On the other hand, for $a\in \R\setminus \Q$ the set $\Z+\tfrac{a}{2}\Z$ is dense in $\R$, so $L$ is (GAH) if and only if the \eqref{ADC3} condition is satisfied.
\end{ex}

By adapting ideas from \cite{AKM,AK19,Ber94,Ber99} it is possible to construct several other interesting examples of globally analytic hypoelliptic operators by choosing the complex numbers $c$ and $q$ conveniently.

\begin{lemma} \label{equivalence}The operator $L = \partial_t + {c} \partial_0 +q$ satisfies the \eqref{ADC} condition if and only if the following condition holds: for all $B>0$ there is a constant $C>0$ such that
	\begin{equation}\label{ADC4}
		|1-e^{\pm 2\pi (im{c} + q)}| \geq Ce^{-B\ell}, \tag{ADC$_4$}
	\end{equation}
for all $\ell \in \tfrac{1}{2}\N_0$, $-\ell \leq m \leq \ell$, $\ell-m\in\N_0$ such that $im{c} + q \notin i\mathbb{Z}$.
\end{lemma}

\begin{proof} If $L$ does not satisfy the \eqref{ADC4} condition, then there exist a constant $B>0$, a sequence $\{\ell_j \}_{j \in \N}$, and indexes $-\ell_j \leq m_j \leq \ell_j$ such that for all $j \in \mathbb{N}$
$$
0< |1-e^{\pm 2\pi (im_j{c} + q)}| < \tfrac{1}{j} e^{-B \ell_j}.
$$
In particular, for $c=a+ib$, with $a,b \in \mathbb{R}$, we have
$$e^{\pm 2\pi(im_j{c} +q)} = e^{2\pi(\emph{\emph{\mbox{Re}}}( q)-m_j{b})}\cdot e^{2\pi i (m_j{a}+\emph{\emph{\mbox{Im}}}(q))} \rightarrow 1, \mbox{ when } j \rightarrow \infty,$$
thus $|\emph{\emph{\mbox{Re}}}( q)-m_j{b}| \rightarrow 0$ and there is a sequence of integer numbers $({k}_j)$ such that $|{k}_j +m_j{a}+\emph{\emph{\mbox{Im}}}(q)| \rightarrow 0$, when $j \rightarrow \infty$. By the Mean Value Theorem we can choose $j$ big enough such that
$$
|1-e^{\pm 2\pi (im_j{c} + q)}| \geq |1-e^{2\pi(\emph{\emph{\mbox{Re}}}( q)-m_j{b})}| \geq e^{-1}2\pi |\emph{\emph{\mbox{Re}}}( q)-m_j{b}|
$$
and
$$
|\sin\left(2\pi({k}_j +m_j{a}+\emph{\emph{\mbox{Im}}}(q))\right)| \geq \pi |{k}_j +m_j{a}+\emph{\emph{\mbox{Im}}}(q)|,
$$
which implies that
\begin{align*} \pi |{k}_j +m_j{a}+\emph{\emph{\mbox{Im}}}(q)| & \leq  |\sin\left(2\pi({k}_j +m_j{a}+\emph{\emph{\mbox{Im}}}(q))\right)| \\
 & \leq  2e^{2\pi (\emph{\emph{\mbox{Re}}}( q)-m_j{b})} |\sin\left(2\pi({k}_j +m_j{a}+\emph{\emph{\mbox{Im}}}(q))\right)| \\ 
  & =  2| \emph{\emph{\mbox{Im}}} (1-e^{\pm 2\pi (im_j{c} + q)})| \\
  & \leq  2 |1-e^{\pm 2\pi (im_j{c} + q)}|.
\end{align*}

Therefore, there is a positive constant $C$ such that for $j$ big enough we have
$$0<|{k}_j + m_j{c} -iq| \leq |\emph{\emph{\mbox{Re}}}(q)-m_j{b}| + |{k}_j+m_j{a}+\emph{\emph{\mbox{Im}}}(q)| \leq \dfrac{C}{j}e^{-B \ell_j}.$$
 
From this we conclude that $L$ does not satisfy the \eqref{ADC} condition. 

On the other hand, assume now that $L$ does not satisfy the  \eqref{ADC} condition, so there is a positive constant $B>0$, a sequence  $({k}_j,\ell_j)$ in $\mathbb{Z}\times \tfrac{1}{2}\N_0$ and indexes $-\ell_j \leq m_j \leq \ell_j$ such that
$$
0<|{k}_j+{c} m_j-iq| \leq \frac{1}{j}e^{-B(|{k}_j|+\ell_j)},
$$
for all $j \in \mathbb{N}$. In particular, $|{k}_j + m_j {a}+\emph{\emph{\mbox{Im}}}(q)| \rightarrow 0$ and $|m_j {b} - \emph{\emph{\mbox{Re}}}(q)| \rightarrow 0$ when $j \rightarrow \infty$. Therefore, taking $j$ big enough we can apply the Mean Value Theorem again and obtain a constant $C>0$ such that
\begin{align*} 
|1-e^{\pm 2\pi (im_j{c} + q)}| & \leq  |1-e^{\pm 2\pi(\emph{\emph{\mbox{Re}}}(q)-m_j{b})}\cos(2\pi(m_j{a}+\emph{\emph{\mbox{Im}}}(q)))| \\
  & \quad +  |e^{\pm 2\pi(\emph{\emph{\mbox{Re}}}(q)-m_j{b})}|\cdot |\sin(2\pi(m_j{a}+\emph{\emph{\mbox{Im}}}(q)))| \\
   & \leq  |1-\cos(2\pi({k}_j+m_j{a}+\emph{\emph{\mbox{Im}}}(q)))| + |1-e^{\pm 2\pi(\emph{\emph{\mbox{Re}}}(q)-m_j{b})}| \\
   & \quad + e^{\pm 2\pi(\emph{\emph{\mbox{Re}}}(q)-m_j{b})} |\sin(2\pi({k}_j+m_j{a}+\emph{\emph{\mbox{Im}}}(q)))| \\
   & \leq  C (|{k}_j+m_j{a}+\emph{\emph{\mbox{Im}}}(q)|+|\emph{\emph{\mbox{Re}}}(q)-m_j{b}|) \\
   & \leq  C e^{-B(|{k}_j|+\ell_j)}
\end{align*}
 and then $L$ does not satisfy the \eqref{ADC4} condition.

\end{proof}

The first three formulations of the analytic Diophantine condition are essentially convenient ways to write the same thing. They are useful tools for producing new examples of operators satisfying (or not satisfying) the conditions to be globally analytic hypoelliptic. The \eqref{ADC4} formulation will be important mainly in the proof of some results in the next section.

\begin{ex}\label{lambdaiZ} 
	When $c=0$, we have $L = \partial_t+{q}$. Thus
$$\mathcal{N} = \{(k,\ell)\in \mathbb{Z}\times \tfrac{1}{2}\N_0; k - i{q}=0 \},$$
and $\mathcal{N} \neq \varnothing$ if and only if $i{q} \in \mathbb{Z}$. 

By Proposition \ref{Ninfinite}, $L$ is not (GAH) when  $iq \in \mathbb{Z}$. On the other hand, when $i {q} \notin \mathbb{Z}$, we have $\mathcal{N} = \varnothing$ and $|1-e^{2\pi{q}}| =C  \neq 0$, therefore
$$
|1-e^{2\pi{q}}| \geq Ce^{-\ell},
$$
for all $\ell \in \tfrac{1}{2}\N_0$, and $L$ is (GAH).
\end{ex}

\bigskip
\section{A class of evolution operators}

Inspired by the works of Hounie \cite{Hou79,Hou82}, Bergamasco \cite{Ber94,Ber99}, Petronilho \cite{Pet11} and so many other researchers that have studied global properties of vector fields on tori, it is natural to inquire if the operator  
$$
\partial_t +c(t) \partial_0,
$$ 
is globally  analytic hypoelliptic on $\mathbb{T}^1 \times \mathbb{S}^3$, when $c \in C^\omega (\mathbb{T}^1)$. 

Let us prove that the answer to this question is negative. For $u \in \mathcal{D}'(\mathbb{T}^1\times \St)$ and $\ell \in \frac{1}{2} \mathbb{N}_0$ we have the following matrix entries of the partial Fourier coefficients of $(\partial_t+c(t)\partial_0)u$: 
$$
\left[\partial_t +imc(t) \right] \widehat{u}(t,\ell)_{mn}, \mbox{ for }  -\ell \leq m,n \leq \ell.
$$

Now consider the sequence of real analytic functions $\{\widehat{u}(t,\ell); \ell \in \frac{1}{2} \mathbb{N}_0\}$ defined by 
$$
	\widehat{u}(t,\ell)_{mn} =
	\begin{cases}
		1, & \mbox{ whenever } m=0; \\
		0, & \mbox{ otherwise.}
	\end{cases}
$$

This sequence defines a distribution $u \in \mathcal{D}'(\mathbb{T}^1\times \St)$ such that $(\partial_t +c(t) \partial_0)u=0$. Moreover, since the set $\{ \ell \in \frac{1}{2}\mathbb{N}_0; \widehat{u}(t,\ell)_{mn}=1\}$ has infinitely many elements, this sequence does not correspond to any real analytic function, which finishes the proof.

\medskip
However, let us recall that in  Examples \ref{ex_const-coef} and \ref{lambdaiZ} we obtained some interesting cases of globally analytic hypoelliptic operators by adding a zero order perturbation to a constant-coefficient vector field defined on $\mathbb{T}^1\times \mathbb{S}^3$. Therefore it seems reasonable to consider the following class of perturbed operators
\begin{equation}\label{Pq}
P \doteq   \partial_t + c(t) \partial_0 + q
\end{equation}
defined on $\mathbb{T}^1\times \mathbb{S}^3$, where $q \in \mathbb{C}$ and $c(t)=a(t)+ib(t)$, with $a,b \in C^\omega(\mathbb{T}^1;\mathbb{R})$ being real-valued analytic functions. 

In the remainder of this article we will also adopt the following notations: $C(t) = A(t)+iB(t)$, where  
$$A(t)= \int_0^{t}a(s)ds \mbox{ \ and \ } B(t)= \int_0^{t}b(s)ds;$$ 
and $c_0 = a_0+ib_0$ where 
\begin{equation}\label{c0}
a_0= \frac{1}{2\pi} \int_0^{2 \pi} a(s)ds \mbox{ \ and \ } b_0= \frac{1}{2\pi} \int_0^{2 \pi} b(s)ds.
\end{equation}

We will also consider the constant-coefficient operator
 \begin{equation}\label{P0q}
   P_0 \doteq  \partial_t + c_0\partial_0 + q,
\end{equation}
and the corresponding set
\begin{align*}
\mathcal{N}_0 & = \left\{({k},\ell) \in \mathbb{Z}\times \tfrac{1}{2}\mathbb{N}_0 ;{k}+c_0 m-i{q}=0, \mbox{ for some } -\ell \leq m \leq \ell, \ \ell-m \in \N_0 \right\}\\ 
& =  \left\{ \ell \in \frac{1}{2}\mathbb{N}_0; \ imc_0+q \in i \mathbb{Z}, \ \textrm{ for some } -\ell \leq m \leq \ell, \ \ell-m \in \N_0 \right\}.
\end{align*}

By Proposition \ref{Ninfinite}, the number of elements of set $\mathcal{N}_0$ is related to the global analytic hypoellipticity of $P_0$. The next result establishes the same type of connection between the set $\mathcal{N}_0$ and the operator $P$.

\begin{prop}\label{Ninfinite2} 
	If $P$ is (GAH), then $\mathcal{N}_0$ is finite. 
\end{prop}

\begin{proof} Suppose by contradiction that $\mathcal{N}_0$ has infinitely many elements. In this case, there are sequences $\{ ({k}_j,\ell_j) \in \mathbb{Z}\times \tfrac{1}{2}\mathbb{N}\}_{j\in \N}$ and $\{ m_j\}_{j\in \N}$, with $-\ell_j \leq m_j \leq \ell_j$, such that
\begin{equation}\label{integer} {k}_j +c_0 m_j -iq=0, \  j \in \mathbb{N}.
\end{equation}
	
For each $j\in\N$, let $t_j \in [0,2\pi]$ and $M_j \in \mathbb{R}$ such that
$$ 
M_j = \int_{0}^{t_j} \big[\mbox{Re}(q)-m_jb(\sigma) \big] d\sigma = \max\limits_{0 \leq t \leq 2\pi} \int_{0}^t \big[ \mbox{Re}(q) - m_jb(\sigma) \big] d\sigma,
$$
and consider the sequence
$$
  \widehat{u}(t,\ell)_{mn} = 
  \begin{cases}
    \exp\left\{\displaystyle \int_{0}^t \big[im_jc(\sigma)+{q} \big] d\sigma -M_j \right\}, & \mbox{if } \ell=\ell_j \mbox{ and } m=n=m_j \\ 
    0 & \mbox{otherwise.}
  \end{cases}
$$

Since $im_jc_0+q \in i \mathbb{Z}$, for all $j \in \N$, all the partial Fourier coefficients $\widehat{u}(\cdot,\ell) \in C^\infty(\mathbb{T}^1)$ and 
$$
|\widehat{u}(t,\ell_j)_{m_j m_j}| = \exp\left\{\int_{0}^t \big[\mbox{Re}(q)-m_jb(\sigma)\big] -M_j  \right\} \leq 1,  \mbox{ for all } t \in [0,2\pi].
$$

Therefore the sequence $\{\widehat{u}(\cdot,\ell); \ell \in\tfrac{1}{2}\mathbb{N}\}$ corresponds to a distribution $u \in \mathcal{D}'(\mathbb{T}^1\times \St)$. Moreover
$$
\widehat{u}(t_j,\ell_j)_{m_j m_j} = 1, \ j \in \mathbb{N}.
$$

Now, note that set $\{\ell_j\}_{j\in \mathbb{N}}$ cannot be bounded. Indeed, if it were bounded, by \eqref{integer}, the set $\{{k}_j\}_{j \in \mathbb{N}}$ would also be bounded and, consequently, the set $\mathcal{N}_0$ would be finite (because $\N$ is discrete). 
	
Hence, $\{\ell_j\}_{j \in \mathbb{N}}$ has infinitely many elements and $u \in \mathcal{D}'(\mathbb{T}^1\times \St) \setminus C^\infty(\mathbb{T}^1\times \St)$. Since
$$
[\partial_t +im_jc(t)+q]\widehat{u}(t,\ell_j)_{m_jm_j} = 0, \ j\in\N,
$$
then $P u=0$ and $P$ is not (GAH). 

\end{proof}

Given the last proposition, the next step is to find out under what conditions (on the perturbation $q$ and on the coefficient $c(t)$) the set $\mathcal{N}_0$ is finite.  

\medskip
Suppose that $u \in \mathcal{D}'(\mathbb{T}^1\times \St)$ is a solution of $P u=f \in C^\omega(\mathbb{T}^1 \times \St)$, then for each $\ell \in\tfrac{1}{2}\mathbb{N}$ and $-\ell \leq m,n \leq \ell$ we have
$$
\widehat{P u}(t,\ell)_{mn} = \big[\partial_t+imc(t)+{q}\, \big]\widehat{u}(t,\ell)_{mn}= \widehat{f}(t,\ell)_{mn},
$$
that is equivalent to 
\begin{equation}\label{conjugation} 
\big[ \partial_t + imc_0+{q}\, \big](e^{im \mathscr{C}(t)}\widehat{u}(t,\ell)_{mn}) = e^{im\mathscr{C}(t)}\widehat{f}(t,\ell)_{mn}.
\end{equation}
where 
$$\mathscr{C}(t) \doteq  C(t) - c_0t.$$

\medskip
All solutions of the ordinary differential equations \eqref{conjugation} are analytical and their expressions depend on the value 
\begin{equation}\label{gamma_m}
	\gamma_m=imc_0+{q} = (\mbox{Re}({q}) -mb_0) + i (ma_0+\mbox{Im}({q})).	
\end{equation}

The next result gives us the form of a solution in terms of a fixed constant.

\begin{lemma}\label{solutionconstant}
	Let $\gamma \in \C$, $g \in C^\infty(\mathbb{T}^1)$ and consider the equation
	\begin{equation}\label{equationconstant}
	\frac{d}{dt}v(t)+\gamma v(t)=g(t).
	\end{equation}
	
	If $\gamma \notin i\Z$ then the equation \eqref{equationconstant} has a unique solution that can be expressed by
	\begin{equation}\label{solutionminus}
	v(t)=\frac{1}{1-e^{-2\pi\gamma}} \int_0^{2\pi} e^{-\gamma s} g(t-s)\, ds,
	\end{equation}
	or equivalently,
	\begin{equation}\label{solutionplus}
	v(t)=\frac{1}{e^{2\pi\gamma}-1} \int_0^{2\pi} e^{\gamma r} g(t+r)\, dr.
	\end{equation}
	
	If $\gamma \in i\Z$ and $\int_0^{2\pi} e^{\gamma s}g(s)\, ds=0$ then we have that
	\begin{equation}\label{solutionzero}
	v(t)=e^{-\gamma t}\int_0^t e^{\gamma s}g(s) \, ds
	\end{equation}
	is a solution of the equation \eqref{equationconstant}.
\end{lemma}
To prove this lemma observe that $E=(1-e^{-2\pi\gamma})^{-1}e^{\gamma t}$ is the fundamental solution of the operator ${d}/{dt}+\gamma$ when $\gamma \notin i\Z$. The equivalence between \eqref{solutionminus} and \eqref{solutionplus} follows from the change of variable $s \mapsto -r+2\pi$. We point out that the solution \eqref{solutionzero} is not unique.

\medskip
Turning back to the problem of determining under what conditions the set $\mathcal{N}_0$ is finite, recall that, by Remark \ref{finiteimpliesempty}, if $\mathcal{N}_0$ is finite, then $\mathcal{N}_0=\varnothing$ and  $\gamma_m \notin i\mathbb{Z}$ for all $m$.

However,
$$
\gamma_m \in i\mathbb{Z} \ \Longleftrightarrow \ \mbox{Re}(q)-mb_0 =0 \mbox{ \ and \ } \mbox{Im}(q) +ma_0 \in \mathbb{Z}.
$$

Thus, when $b_0\neq 0$, this equivalence is true if $m=\mbox{Re}(q)/{b_0} \in \frac{1}{2}\mathbb{Z}$ and $\mbox{Im}(q)+\mbox{Re}(q)\frac{a_0}{b_0} \in \mathbb{Z}$. When $b_0=0$, this equivalence is true if $\mbox{Re}(q)=0$ and $\mbox{Im}(q) \in \mathbb{Z}+\frac{a_0}{2} \mathbb{Z}$.

Therefore, we conclude that $\gamma_m \notin i\mathbb{Z}$ for all $m$ ($\mathcal{N}_0=\varnothing$) if and only if one of the following conditions holds: 
\begin{align}
	& b_0\neq0 \mbox{ and either } \frac{\mbox{Re}(q)}{b_0} \notin \tfrac{1}{2}\mathbb{Z} \mbox{ or } \mbox{Im}(q)+\mbox{Re}(q)\frac{a_0}{b_0} \notin \mathbb{Z}. \tag{C1} \label{C1} \\[1mm]
	& b_0=0  \mbox{ and either } \mbox{Re}(q) \neq 0 \mbox{ or } \mbox{Im}(q) \notin \mathbb{Z}+\frac{a_0}{2} \mathbb{Z}. \tag{C2} \label{C2}
\end{align}

\bigskip
\subsection{The Nirenberg-Treves condition ($\mathcal{P}$)} \

Now that we have established conditions on $c_0$ and $q$, we can state our first result about the global analytic hypoellipticity of the operator $P$. We will begin by addressing the case where $b(t)=\mbox{Im}\, c(t)$ does not change sign, that is, when the Nirenberg-Treves condition ($\mathcal{P}$) holds.

\begin{thm}\label{Lisgah} 
Let $q \in \mathbb{C}$ and $c(t)\in C^\omega(\mathbb{T}^1)$. If
 	\begin{enumerate}
		\item[i.] ${\rm Im}\, c(t) \not\equiv 0$ and does not change sign; 
		\item[ii.] $P_0 = \partial_t+c_0 \partial_0+{q}$ satisfies the \eqref{ADC} condition; 
		\item[iii.] $q\in\C$ satisfies condition \eqref{C1}.
	\end{enumerate}
Then $P = \partial_t +c(t)\partial_0 +{q}$ is (GAH). 
\end{thm}

\begin{proof}
Writing $c(t)=a(t)+ib(t)$ and $c_0=a_0+ib_0$, since $b$ does not change sign and is not identically zero, then $b_0>0$ or $b_0<0$. Up to a change of variable, we can assume without loss of generality that $b_0<0$  
		
By condition \eqref{C1}, for each $\ell \in\tfrac{1}{2}\mathbb{N}$ and  $-\ell \leq m,n \leq \ell$, the equation \eqref{conjugation} has only one periodic smooth solution, which can be written in the form 
\begin{equation}\label{solution_m>0}
\widehat{u}(t,\ell)_{mn} = (1-e^{-2\pi\gamma_m})^{-1}\int_{0}^{2\pi} e^{-{q} s}e^{-im \int_{t-s}^t c(\sigma)d\sigma}\widehat{f}(t-s,\ell)_{mn} ds,
\end{equation}
when $m \geq 0$, or
\begin{equation}\label{solution_m<0}
\widehat{u}(t,\ell)_{mn} = (e^{2\pi\gamma_m}-1)^{-1}\int_{0}^{2\pi} e^{qs}e^{-im \int_t^{t+s} c(\sigma) d\sigma} \widehat{f} (t+s,\ell)_{mn} ds.
\end{equation}
when $m \leq 0$. These expressions can be obtained from \eqref{solutionminus}, when $m \geq 0$, and from \eqref{solutionplus}, when $m \leq 0$.

Since the two equivalent expressions above provide the only solutions of the ordinary differential equations \ref{conjugation}, from this point the proof consists in proving that these partial Fourier coefficients satisfy the condition of decay at infinity given by Proposition \ref{functionanaly} and Remark \ref{noderivative}.

Let us make estimates for the case where $m\geq 0$ (the other case is completely analogous and follows the same steps).

Since we are assuming $b_0<0$, and $b$ does not change sign, then $b(\sigma)\leq 0$, for all $ \sigma \in [0,2\pi]$. Moreover, since we are assuming that $m \geq 0$, we have 
$$
\left|\exp{\left(-im \int_{t-s}^t c(\sigma)d\sigma\right)}\right| = \exp\left(m\int_{t-s}^t b(\sigma)d\sigma\right) \leq 1, \ \mbox{ for all } t,s \in [0,2\pi].
$$

Recall that $f \in C^\omega(\mathbb{T}^1\times \St)$, therefore, by Proposition \ref{functionanaly}, there are constants $B>0$ and $M>0$ such that
$$
|\widehat{f}(\cdot , \ell)_{mn}| \leq M e^{-B \ell},
$$
for all $\ell \in \frac{1}{2} \mathbb{N}_0$ and $-\ell \leq m,n \leq \ell$. 

Now, if $C >0$ is such that $|e^{q s}| <C$, for all $ s \in [0,2\pi]$, then we have
\begin{align*}
\left| \widehat{u}(t,\ell)_{mn} \right| & \leq  \left| e^{2\pi\gamma_m}-1\right|^{-1} \int_{0}^{2\pi}\! \left| e^{qs}\right|   \left| \exp\left(-im \! \int_t^{t+s} c(\sigma) d\sigma\right) \right|  \left| \widehat{f} (t+s,\ell)_{mn}\right| ds \\[2mm]
& \leq 2\pi M C e^{-B \ell} \left| e^{2\pi\gamma_m}-1\right|^{-1},
\end{align*}
for all $\ell \in \frac{1}{2} \mathbb{N}_0$ and $-\ell \leq m,n \leq \ell$. 

Finally, since $P_0$ satisfies the \eqref{ADC} condition, choosing $\widetilde{B}>B$, by Lemma \ref{equivalence}, there is a constant $\widetilde{M}>0$ such that
$$
|1-e^{-2\pi\gamma_m}| \geq \widetilde{M} e^{-\widetilde{B}{\ell}},
$$
for all $\ell \in \tfrac{1}{2}\N_0$, $-\ell \leq m \leq \ell$, $\ell-m\in\N_0$ such that $\gamma_m \notin i\mathbb{Z}$. 

Thus,
$$
|\widehat{u}(t,\ell)_{mn}| \leq  K e^{-(\widetilde{B}-B){\ell}},
$$
for all $t \in [0,2\pi]$, $\ell \in \frac{1}{2} \mathbb{N}_0$ and $-\ell \leq m,n \leq \ell$ with $m\geq 0$. 

It follows from Proposition \ref{functionanaly} and Remark \ref{noderivative} that $u \in C^\omega(\mathbb{T}^1\times \St)$ and therefore $P$ is (GAH). 

\end{proof}

\begin{ex}\label{averagen} 
	Let $n\neq 0$ be an integer number and set $b(t) = \sin(t)+n$. Then $b\in C^\omega(\mathbb{T}^1)$ does not change sign and $b_0 = n$. Taking $q\in\mathbb{C}$ with ${\rm Re}(q) \notin \mathbb{Q}$, there is a constant $C>0$ such that $|b_0 m - \emph{\emph{{\rm Re}}}({q})| \geq C>0$ for all $m \in \frac{1}{2}\mathbb{Z}$. Thus,  for any real-valued analytic function $a(t)$ we have
	$$| ({k}+a_0m+{\emph{\mbox{Im}}}({q}))+i(b_0m-{\emph{\mbox{Re}}}({q}))| \geq |b_0 m -{\emph{\mbox{Re}}}({q})| \geq C.$$

	Therefore, $P_0$ satisfies the \eqref{ADC}  condition and ${q}$ satisfies \eqref{C1}. It follows from Theorem \ref{Lisgah} that $P= \partial_t+\big(a(t)+i(\sin(t)+n)\big)\partial_0 + q$ is (GAH). 
\end{ex}


\begin{thm}\label{equivalence2} If $b\equiv 0$, then $P = \partial_t +a(t)\partial_0 +{q}$ is (GAH) if and only if $P_0 = \partial_t +a_0\partial_0 +{q}$  is (GAH).
\end{thm} 

\begin{proof}
Let $\mathscr{A}(t) \doteq  \displaystyle\int_0^{t}a(s)ds - a_0t$ and
define
$$\Psi_a u(t,x) \doteq  \sum_{\ell \in \frac{1}{2} \N_0} (2\ell +1) \sum_{m,n=1}^{2\ell+1} e^{im\mathscr{A}(t)}\widehat{u}(t,\ell)_{mn}\textsf{t}^\ell(x)_{nm}.
$$

In Propositions 4.6 and 4.7 in \cite{KMR19} it was proved that the map $\Psi_a$ is an automorphism of $\mathcal{D}'( \mathbb{T}^1 \times \St)$ and of $C^\infty(\mathbb{T}^1\times \St)$, and that 
$$P_0 \circ \Psi_a = \Psi_a \circ P.$$ 

Since $|e^{im\mathscr{A}(t)}|=1$, $\Psi_a$ is also an automorphism of $C^\omega(\mathbb{T}^1\times \mathbb{S}^3)$ (see Proposition \ref{functionanaly}). Thus $P$ is (GAH) if and only if $P_0$ is (GAH).

\end{proof}

\subsection{Building singular solutions}

Let us consider the missing case: $b \not\equiv 0$ and $b$ changes sign. We will prove that $P$ is not (GAH) following the technique of building singular solutions introduced by A. Bergamasco and used by several authors (for example, see: \cite{AGK18,AGKM18,AK19,AKM,Ber99,bergamasco2017existence,BZ05}). 

In our case, ``to build a singular solution" means to present a suitable real analytic function $f$ and a distributional solution $u$ of $Pu=f$ that it is not a real analytic function. The main difference between our construction and those used in most of the references previously cited is that we cannot use cut-off functions, which makes the construction process more delicate.

To start building the singular solution, observe that the zeros of the real analytic function $b$ are isolated. Since $b$ is $2\pi$-periodic and changes sign, we can assume, without loss of generality, that $b$ changes sign from minus to plus at the point $t_0=0$ and that $b_0 \leq 0$.

Thus $b$ is strictly positive on some open interval $]0,s[$ and there is $t^*\in \ ]0,2\pi[$ such that 
$$M\doteq  B(t^*) = \max_{t \in[0,2\pi]} B(t) >0.$$ 
 
\medskip
Define a $2\pi$-periodic real analytic function by 
$$\psi(t) = M+K(1-\cos(t))+i(a(0)\sin(t)-A(t^*)),$$ 
where $K>0$ is a constant that we will choose later. 

\begin{prop}\label{f_analytic}
Let $d_{\ell} = \left(1-e^{-2\pi(i\ell c_0+{q})}\right), \ \ell \in \frac{1}{2}\mathbb{N}$, and consider the sequence of functions \begin{equation}\label{f_singular_case}
\widehat{f}(t,\ell)_{m n} = 
\begin{cases}
d_{\ell} e^{-\ell \psi(t)}, & \mbox{if } m=n=\ell; \\
0								  & \mbox{otherwise.}
\end{cases}
\end{equation}
Then $\{\widehat{f}(\cdot,\ell)_{mn}\}$ is the sequence of partial Fourier coefficients of a real analytic function on $\TS$.
\end{prop}

\begin{proof}
Notice that
$$
|d_{\ell}| \leq 1+e^{-2\pi(\emph{\emph{\mbox{Re}}}({q})-\ell b_0)} = 1+e^{-2 \pi \emph{\emph{\mbox{Re}}}({q})}e^{\ell 2\pi b_0} \leq 1+e^{-2\pi \emph{\emph{\mbox{Re}}}({q})}=C.
$$

Thus,
$$
|\widehat{f}(t,\ell)_{\ell \ell}| = |d_{\ell}| e^{-\ell \emph{\emph{\mbox{Re}}}(\psi(t))} \leq C e^{-\ell(M+K(1-\cos(t)))} \leq C e^{-M \:\cdot\: \ell}
$$
and all the functions $\widehat{f}(\cdot,\ell)_{m n} \in C^\omega(\mathbb{T}^1)$. 

It follows from Proposition \ref{functionanaly} that
the sequence $\{\widehat{f}(t,\ell)_{mn}\}$ defines a function $f \in C^\omega(\mathbb{T}^1\times \mathbb{S}^3)$. 

\end{proof}

The next step is to construct a distribution $u \in \mathcal{D}'(\mathbb{T}^1\times \mathbb{S}^3)$, satisfying $Pu =f$, that is not a real analytic function.

By Remark \ref{finiteimpliesempty} and Proposition \ref{Ninfinite}, we can assume that $\mathcal{N}_0 = \varnothing$. So, the equation
$$
[\partial_t + imc_0+{q}](e^{im \mathscr{C}(t)}\widehat{u}(\:\cdot\:,\ell)_{mn}) = e^{im\mathscr{C}(t)}\widehat{f}(\:\cdot\:,\ell)_{mn}
$$
has only one real analytic periodic solution when $m=n=\ell$, which is given by
\begin{align*} 
\widehat{u}(t,\ell)_{\ell \ell} &=  d_{\ell}^{-1} \int_{0}^{2\pi} e^{-{q} s}e^{-i \ell \int_{t-s}^t c(\sigma)d\sigma} \widehat{f}(t-s,\ell)_{\ell \ell} ds \\
 & =  \int_{0}^{2\pi} e^{-{q} s} e^{-\ell[ \psi(t-s) + i (C(t) - C(t-s))]}ds.
\end{align*}

Let us denote $\Phi(t,s) \doteq \psi(t-s)+i( C(t)-C(t-s))$ and
$$ \varphi(t,s) \doteq - \mbox{Re}(\Phi(t,s)) = B(t)-B(t-s)-M-K(1-\cos(t-s)).$$

\begin{lemma}\label{phi<0}
	There exists a constant $K>0$ such that $\varphi(t,s) \leq 0$, for all $t,s \in [0,2\pi]$.
\end{lemma}

\begin{proof} We will split this proof in three cases.

\medskip \noindent {\bf Case 1:} There exists $\delta_2>0$ such that $\varphi(t,s) \leq 0$ for all $t,s \in [0,2\pi]$ with $2\pi -\delta_2 <|t-s| \leq 2\pi$. 

\noindent We have $|t-s| = 2\pi$ if and only if $(t,s)=(2\pi,0) \textrm{ or } (t,s) = (0,2\pi)$. Since
\begin{align*}
	\varphi(2\pi,0) &= B(2\pi)-B(2\pi)-M-K(1-\cos(2\pi)) = -M<0 \mbox{ and}\\
	\varphi(0,2\pi) &= B(0)-B(-2\pi)-M-K(1-\cos(-2\pi)) = 2\pi b_0-M<0,
\end{align*}
the desired result follows by continuity.

\bigskip \noindent {\bf Case 2:} There is $\delta_1>0$ such that $\varphi(t,s) \leq 0$ for all $t,s \in [0,2\pi]$ with $|t-s|<\delta_1$.	

For $t=s$ we have
	$$\varphi(t,t) = B(t)-B(0)-M-K(1-\cos(0))=B(t)-M\leq 0$$
and for $t \neq s$ in $[0,2\pi]$ we have
$$
\varphi(t,s) \leq 0 \iff \frac{B(t)-B(t-s)-M}{1-\cos(t-s)} \leq K.$$

Let us prove that for each fixed $t \in [0,2\pi]$ the function 
$$g(u)\doteq \frac{B(t)-B(u)-M}{1-\cos(u)}$$ 
has an upper bound on some neighborhood of $u=0$, which does not depend on $t$. 

Since $B(t)-M \leq 0$, for all $t$, $1-\cos(u)>0$, for $u \neq 0$, $B(0)=0$ and $B'(0)=b(0)=0$ then, by Taylor's formula, we have
\begin{align*}
g(u) \leq \frac{B(u)}{1-\cos(u)} & = \frac{1}{1-\cos(u)} \left[ B(0)+B'(0)u + \frac{B''(0)}{2}u^2+R_2(u)\right] \\
& =  \frac{B''(0)}{2}\frac{u^2}{1-\cos(u)}+\frac{R_2(u)}{1-\cos(u)}
\end{align*}
with $\lim\limits_{u \rightarrow 0} R_2(u)/u^2=0.$

\medskip
Since 
$$\displaystyle\lim_{u \rightarrow 0} \frac{u^2}{1-\cos(u)}=2 \ \mbox{ and }  \lim_{u \rightarrow 0} \frac{R_2(u)}{1-\cos(u)} = \lim_{u \rightarrow 0} \frac{R_2(u)}{u^2} \cdot \frac{u^2}{1-\cos(u)}  =0,
$$
hence, given $K_1 > B''(0),$ there is $\delta_1>0$ such that $g(u) \leq K_1$ if $|u|<\delta_1.$ 

Therefore
$$
	\varphi(t,s) = [B(t)-B(t-s)-M-K_1(1-\cos(t-s))] \leq 0, \mbox{ if } |t-s|<\delta_1.
$$

\bigskip \noindent {\bf Case 3:}  There exists $K>K_1$ such that $\varphi(t,s) \leq 0$, for all  $t,s \in [0,2\pi]$ with $\delta_1 \leq |t-s| \leq 2\pi-\delta_2$.

Notice that $1-\cos(t-s)>0$,  $(t,s)\in R = \{(t,s) \in [0,2\pi]^2 ;\delta_1 \leq |t-s| \leq 2\pi-\delta_2 \}$, and
$$\rho\doteq\min \{1-\cos(u); \delta_1\leq u \leq 2\pi-\delta_2\}>0.$$ 

Given any $K\geq K_1$, we have that, for any $(t,s) \in R$,
$$
	\varphi(t,s) \leq B(t)-B(t-s)-\rho K \leq 0 \ \Leftrightarrow \ K \geq \frac{B(t)-B(t-s)}{\rho}.
$$

Now, to obtain the last inequality we may increase $K$, if necessary. 
 
\end{proof}

\begin{prop}\label{u_distr}
The sequence
\begin{equation}\label{u_solution}
\widehat{u}(t,\ell)_{mn}= 
\begin{cases}
\displaystyle \int_{0}^{2\pi} e^{-{q} s} e^{-\ell \Phi(t,s)}ds, & \mbox{if } m=n=\ell; \\
0 & \mbox{otherwise},
\end{cases}
\end{equation}
with 
$$\Phi(t,s) \doteq \psi(t-s)+i( C(t)-C(t-s)), \mbox{ for } t,s \in [0,2\pi],$$ 
corresponds to a distribution $u \in \mathcal{D}'(\TS)$, which is a solution of $P u=f,$ with $f \in C^\omega(\TS)$ defined in \eqref{f_singular_case}.
\end{prop}

\begin{proof}
By Lemma \ref{phi<0}, $\varphi(t,s)= -{\rm Re}(\Phi(t,s)) \leq 0$, for all $t,s \in [0,2\pi]$, then
$$
|\widehat{u}(t,\ell)_{\ell \ell}| \leq \int_{0}^{2\pi} e^{-{\rm Re}({q})s} e^{\ell \varphi(t,s)}ds \leq  2\pi e^{|2\pi{\rm Re}(q)|}
$$
and the sequence $\{\widehat{u}(\:\cdot\:,\ell)\}$ defines a distribution that, by construction, is a solution of $P u=f,$ with $f$ defined in \eqref{f_singular_case}. 

\end{proof}

\begin{prop}\label{u_not_analytic}
	There is no real analytic function defined on $\TS$ whose sequence of partial Fourier coefficients is given by \eqref{u_solution}.	
\end{prop}

\begin{proof}
	By proving that the sequence of numbers $\{\widehat{u}(t^*,\ell)\}_{\ell}$ does not decay exponentially, we will prove that the distribution $u \in \mathcal{D}'(\TS)$, defined in the last proposition, cannot be a real analytic function. 
	
First, observe that for $s \in [0,2\pi]$, we have
\begin{align*}
\Phi(t^*,s) =	& \big[B(t^* -s)+K(1-\cos(t^* -s))\big] + i \big[a(0)\sin(t^* -s)-A(t^* -s)\big],
\end{align*}
in particular, $\Phi(t^*,t^*)=0$. \vspace{0.3cm}
	
\noindent Since we are assuming that $b$ is real analytic and changes sign from $-$ to $+$ at $t_0=0$, there exists $\delta^*>0$ such that $B(u)>0$ for all $u$ with $0<|u|<\delta^*$. In particular,  
$$B(t^*-s)>0, \mbox{ whenever }  0<|t^*-s|<\delta^*,$$ 
which implies that 
$${\rm Re}(\Phi(t^*,s)) >0,  \mbox{ in } 0<|t^*-s|<\delta^*.$$
Note that, since $t^* \in (0,2\pi)$, we can choose $\delta^*$ such that $(t^*-\delta^*,t^*+\delta^*) \subset [0,2\pi]$.

In the region $|t^*-s|\geq \delta^*$, we have $1-\cos(t^*-s)\neq0$. Setting 
$$\rho^*=\min\{B(t^*-s); |t^*-s|\geq \delta^*, s \in [0,2\pi]\},$$
we have 
$$
{\rm Re}(\Phi(t^*,s)) = B(t^*-s)+K(1-\cos(t^*-s)) \geq \rho^*+K(1-\cos(t^*-s)).
$$

\noindent However,
$$
\rho^*+K(1-\cos(t^*-s))>0  \iff K> \frac{-\rho^*}{1-\cos(t^*-s)}.
$$

\noindent Therefore, by increasing $K$, if necessary, we have that 
\begin{equation}\label{phi(t*,s)>0}
{\rm Re}(\Phi(t^*,s))  >0 \mbox{ for all } s \in [0,2\pi], \mbox{with } s\neq t^*.
\end{equation}

\medskip
\noindent The next step is to study the asymptotic behavior of the sequence
\begin{equation}\label{u=I+J}
\widehat{u}(t^*,\ell)_{\ell \ell} = \int_{0}^{2\pi}e^{-qs}e^{- \ell  \Phi(t^*,s)}ds = I_{\ell,\delta}+J_{\ell,\delta},
\end{equation}
where

\begin{align*}
I_{\ell,\delta} = & \int_{R_\delta}e^{-qs}e^{-\ell \Phi(t^*,s)}ds \\
				   = & \int_{\widetilde{R}_\delta}e^{-q(t^*-\sigma)} e^{-\ell[ B(\sigma)+K(1-\cos(\sigma)) +i(a(0)\sin(\sigma)-A(\sigma))]}d\sigma
\end{align*}
and
\begin{align*}
J_{\ell,\delta} = & \int_{(R_\delta)^c}e^{-qs} e^{-\ell \Phi(t^*,s)}ds \\ 
					   = & \int_{(\widetilde{R}_\delta)^c} e^{-q(t^*-\sigma)} e^{-\ell[ B(\sigma)+K(1-\cos(\sigma)) +i(a(0)\sin(\sigma)-A(\sigma))]}d\sigma,
\end{align*}
with $R_\delta = \{s \in [0,2\pi]; |t^*-s|<\delta \}$ and $\widetilde{R}_\delta = \{ \sigma \in [t^*-2\pi, t^*]; |\sigma|<\delta \}$. 
Now, if $|\sigma|\geq \delta$ and $\rho = \min_{|\sigma|\geq \delta}B(\sigma)$ then
$$
B(\sigma)+K(1-\cos(\sigma)) \geq \rho+K(1-\cos(\delta)) = C_{K,\delta}>0
$$
therefore
\begin{equation}\label{JLdelta}
|J_{\ell,\delta}| \leq C \int_{(\widetilde{R}_\delta)^c}  e^{-\ell(\rho+K(1-\cos(\delta)))}d\sigma \leq \widetilde{C} e^{-\ell C_{K,\delta}}.
\end{equation}

\medskip
\noindent Now let us analyze the integral $I_{\ell,\delta}$. Notice that
$$
I_{\ell,\delta} = \int_{\widetilde{R}_\delta}e^{q(t^*-\sigma)} e^{-\ell  \phi(\sigma)}d\sigma,
$$
where
\begin{align*}
\phi(\sigma)  & = (B(\sigma)+K(1-\cos(\sigma)))+i(a(0)\sin(\sigma)-A(\sigma)) \\ 
& = -iC(\sigma)+K(1-\cos(\sigma))+ia(0)\sin(\sigma).
\end{align*}

\noindent Replacing $\sigma$ by the complex variable $z=\sigma+i\tau$ in the above expression, we obtain a function $\phi(z)$ holomorphic on some square $\{z=\sigma+i\tau \in \mathbb{C}; |\sigma|,|\tau|<\delta\}$. Thus
\begin{align*}
\phi'(z) & =  -i(a(z)+ib(z))+K\sin(z)+ia(0)\cos(z), \mbox{ and} \\
\phi''(z) & =  -i(a'(z)+ib'(z))+K\cos(z)-ia(0)\sin(z),
\end{align*}
and
\begin{align*}
\phi(0) & =  -i(A(0)+iB(0)) = 0, \\
\phi'(0) & = -ia(0)+ib(0)+ia(0)=0, \mbox{ and} \\
\phi''(0) & =  -i(a'(0)+ib'(0))+K = K-ic'(0) \neq 0, \mbox{ if $K$ is big enough.}
\end{align*}

\noindent Since $\phi(z)$ is holomorphic in some small neighborhood of the origin, and $z=0$ is the only critical point of $\phi(z)$ on the square $\{z \in \mathbb{C}; |z|< \delta\}$, for $\delta$ sufficiently small, then the function $\phi(z)$ fits the hypothesis of Theorem 2.8 in \cite{Sjostrand} and there exists $\varepsilon >0$ such that
\begin{equation}\label{ILdelta}
I_{\ell,\delta} = \sqrt{2\pi} \dfrac{1}{\sqrt{2\ell}}+\dfrac{1}{\varepsilon}e^{-2\ell \varepsilon},
\end{equation}
for all $\ell \in \frac{1}{2} \mathbb{N}$.

\noindent It follows from \eqref{u=I+J}, \eqref{JLdelta} and \eqref{ILdelta} that 
\begin{align*}
|\widehat{u}(t^*,\ell)|  = & |I_{\ell,\delta}+J_{\ell,\delta}| \\ 
\geq & \sqrt{2\pi}\frac{1}{\sqrt{2\ell}}+\frac{1}{\varepsilon}e^{-2\ell \varepsilon}-2\pi e^{-\ell C_{K,\delta}} \\
= & O \left( \frac{1}{\sqrt{\ell}}\right), 
\end{align*}
when $\ell \to \infty$.

\noindent Since $\widehat{u}(t^*,\ell)$ does not decay exponentially, the sequence of functions $\{\widehat{u}(t,\ell)_{mn}\}$ does not correspond to any real analytic function on $\TS$, what finishes the proof.

\end{proof}

Finally, from Propositions \ref{f_analytic}, \ref{u_distr} and \ref{u_not_analytic} we have the following theorem.

\begin{thm}\label{notgah} 
	If $b$ changes sign, then $P= \partial_t + (a(t)+ib(t))(t) \partial_0 + q$ is not (GAH), for all ${q} \in \mathbb{C}$.
\end{thm}

\section{Final Remarks and examples}

In this section we present some improvements in the results of the previous sections and also build some examples of global analytic hypoelliptic operators.

Our first result establishes a relation between the global analytic hypoellipticity of $P$ and $P_0$ defined in \eqref{Pq} and \eqref{P0q} respectively.

\begin{thm}\label{relation} If $P$ is (GAH), then $P_0$ is also (GAH).	
\end{thm}

\begin{proof} Suppose that $P_0$ is not (GAH) then, by Proposition \ref{gahconst}, either the set $\mathcal{N}_0$ has infinitely many elements or $P_0$ does not satisfy the \eqref{ADC} condition. 

If $\mathcal{N}_0$ is infinite, by Proposition \ref{Ninfinite2},  $P$ is not (GAH) and we are done. 

Now, if $\mathcal{N}=\varnothing$ and $P_0$ does not satisfy the \eqref{ADC} condition, then there exist a constant $B>0$, a sequence $\{({k}_j,\ell_j) \in \mathbb{Z}\times \frac{1}{2}\mathbb{N}; j\in\N\}$ and indexes $-\ell_j \leq m_j \leq \ell_j$ such that
\begin{equation}\label{conv-to-zero}
	0  <  |({k}_j+{\rm Im}(q)+m_j a_0)-i(\emph{\emph{\mbox{Re}}}({q})-m_j b_0)| 
	\leq \frac{1}{j} e^{-B(|{k}_j|+{\ell_j})}
\end{equation}
for all $ j \in \N$.

If the sequence $\{m_j\}_{j\in\N}$ were bounded, since the right hand side in \eqref{conv-to-zero} goes to zero and the midterm assumes values only in a discrete set, then we must have $m_j = 0$ for big enough values of $j$, which would contradict the fact that the midterm is positive for all $j$. Hence we have $ m_j \rightarrow \infty $.

However, the convergence to zero in \eqref{conv-to-zero} also implies that ${\rm Re}({q})-m_j b_0 \rightarrow 0$ when $m_j \to \infty$ and this implies that ${\rm Re}({q})=b_0=0$. Therefore $b\equiv0 $ or $b$ changes sign.

If $b\equiv 0$, $P$ is not (GAH) by Proposition \ref{equivalence2} and if $b$ changes sign, $P$ is not (GAH) by Theorem \ref{notgah}. In any case, we concluded that $P$ is not (GAH).

\end{proof}

In view of the last result, we can improve Theorem \ref{Lisgah} by getting an equivalence.

\begin{thm}\label{mainthm} 
	Given $q \in \mathbb{C}$ and $c\in C^\omega(\mathbb{T}^1)$, the operator $P = \partial_t +c(t)\partial_0 +{q}$ is (GAH) if and only if the three conditions bellow are satisfied:
	\begin{enumerate}
		\item[i.] ${\rm Im}\, c(t)$ does not change sign;
		\item[ii.] $q\in\C$ satisfies either \eqref{C1} or \eqref{C2};		 
		\item[iii.] $P_0 = \partial_t+c_0 \partial_0+{q}$ satisfies the \eqref{ADC} condition.
	\end{enumerate}
\end{thm}

\begin{remark} \label{last-rem}
	Although uncommon, in the last theorem, we included the zero constant function among those functions that do not change sign.
	
	In view of Proposition \ref{gahconst}, when ${\rm Im}\, c(t)\equiv 0$, conditions $ii.$ and $iii.$ together are equivalent to say that $P_0$ is (GAH). 
	
	Let us prove that conditions $i.$ and $ii.$ together imply condition $iii.$, when \linebreak ${\rm Im}\, c(t)\not\equiv 0$.
	
	Recall that condition $ii.$ means that, for all $m \in \frac{1}{2}\mathbb{Z}$ and $k \in \mathbb{Z}$, we have
$$
\gamma_{k,m}=k+c_0m+iq = (k+a_0m+{\rm Im}(q))+i(mb_0-{\rm Re}(q)) \neq 0.
$$
Besides, since $b \not\equiv 0$ and does not change sign, we have $b_0 \neq 0$. Now let us split the proof in two cases: 
\begin{enumerate}
	\item ${\rm Re}(q)/b_0 \in \frac{1}{2} \mathbb{Z}$. \\[3mm]	
	In this case, we must have ${\rm Re}(\gamma_{k,m}) \neq 0$ for $m={\rm Re}(q)/b_0$, which implies $a_0m+{\rm Im}(q) \notin \mathbb{Z}$ and therefore
	$$
	\min_{k\in\Z}\left\{|k + a_0m+{\rm Im}(q)|\right\}=C_1>0,
	$$
	and so 
	$$|\gamma_{k,m}| \geq C_1 \geq C_1 e^{-B(|k|+|m|)}.$$
	
	\
	
	\item ${\rm Re}(q)/b_0 \notin \frac{1}{2} \mathbb{Z}$.\\[3mm]	
	If $m=0$, then 
	$$|b_0m-{\rm Re}(q)| = |{\rm Re}(q)|>0,$$ 
	and if $m \in \frac{1}{2}\mathbb{Z}\setminus \{0\}$, then $mb_0-{\rm Re}(q) \neq 0$ and
	$$
	\min \left\{|mb_0 - {\rm Re}(q)|; m \in \tfrac{1}{2} \mathbb{Z}\setminus\{0\} \right\}=\tilde{C}_2>0.
	$$
	Therefore, for all $k$ and $m$, we have 
	$$|\gamma_{k,m}| \geq |b_0m-{\rm Re}(q)| \geq \min\{\tilde{C}_2,|{\rm Re}(q)|\} = C_2 \geq C_2 e^{-B(|k|+|m|)}.$$
\end{enumerate}

In any case, $P_0$ satisfies the \eqref{ADC}  condition.
\end{remark}

\medskip
In view of works \cite{Ber94,KMR19b,KMR19c}, it is natural to consider operators in the form
\begin{equation}\label{Pq(t,x)}
	P = \partial_t+c(t)\partial_0+ q(t,x), 
\end{equation}
where $c\in C^\omega(\mathbb{T}^1)$ and $q \in C^\omega(\TS)$.

\medskip
In the study of this class of operators we are led to consider also the constant-coefficient operator
\begin{equation}\label{P0q}
P_{00} \doteq  \partial_t + c_0\partial_0 + q_0,
\end{equation}
where 
$$c_0 = \frac{1}{2\pi} \int_0^{2 \pi} c(s)ds \mbox{ \ and \ } q_0 = \frac{1}{2\pi} \int_{\St} \int_0^{2 \pi}q(s,x)dsdx,$$
that is, $c_0$ and $q_0$ are the averages of functions $c(t)$ and $q(t,x)$, respectively.

\begin{prop}\label{nonconstant} Given $q \in C^\omega(\TS)$, assume that there is an analytic function $Q \in C^\omega(\mathbb{T}^1\times\St)$ such that $$(\partial_t+c(t)\partial_0){Q} = {q} - {q}_0.$$
 
 Then we have 
 $$
 P \circ e^{-Q} = e^{-Q} \circ P_{00},
 $$
 in both $\mathcal{D}'(\TS)$ and in $C^\omega(\TS)$. Hence, $P$ is (GAH) if and only if $P_{00}$ is (GAH).
\end{prop} 

The proof of this proposition is an adaptation of the corresponding result in section 5 of  \cite{KMR19b}.

\begin{ex}\label{example2} Let $Q:\emph{SU}(2) \rightarrow \mathbb{C}$ be the function given by 
$$
Q(a,b)=-2a\bar{b}, \mbox{ where } (a,b) \in \mathbb{C}^2 \mbox{ and } |a|^2+|b|^2 = 1.
$$ 
In Euler-angles coordinates we have 
$$Q(\phi,\theta,\psi)=ie^{i\psi}\sin(\theta),$$ 
hence 
$$\partial_0 Q = i \frac{\partial}{\partial \psi} Q = -Q.$$ 
If we consider $q:\TS \rightarrow \mathbb{C}$ defined by 
$$q(t,a,b) =2(\alpha+i\sin(t))a\bar{b}+i/2,$$ 
then $q_0 = i/2$ and 
$$(\partial_t+(\alpha+i\sin(t))\partial_0)Q = q-q_0.$$ 
In this way, the operator
$$
P = \partial_t+(\alpha+i\sin(t))\partial_0+q
$$
is conjugated with the operator $P_{00} = \partial_t+(\alpha+i\sin(t))\partial_0+i/2$, which is not (GAH) because $b(t)=\sin(t)$ changes sign. Therefore $P$ is not (GAH).
\end{ex}

\begin{ex} Let $c(t)=a(t)+i(\sin(t)+n)$ be a real analytic function on $\mathbb{T}^1$, where $n \in \mathbb{Z}\setminus \{0\}$  and $r\in\mathbb{C}$ with ${\rm Re}(r) \notin \mathbb{Q}$.
	
	Define a complex-valued function on $\TS$ by   
	$$q(t,a,b)=2c(t)a\bar{b}+r, \ (t,a,b)\in \mathbb{T}^1 \times \emph{SU}(2).$$ 

Hence, $q_0=r$ and $ (\partial_t+c(t)\partial_0)Q = q-q_0,$ where $Q$ is as in Example \ref{example2}. 

It follows from Example \ref{averagen} that the operator $$P_{00} = \partial_t+c(t)\partial_0+r  \mbox{ is (GAH) }$$ and, by Proposition \ref{nonconstant}, $P$ is also (GAH).	
\end{ex}

\section*{Acknowledgments}
This study was financed in part by the Coordenação de Aperfeiçoamento de Pessoal de Nível Superior - Brasil (CAPES) - Finance Code 001.

\end{document}